\documentclass[a4 paper, 10pt]{article}
\usepackage{amsmath}
\usepackage{latexsym}
\usepackage{amsfonts}
\usepackage{amssymb}
\usepackage{amsthm}
\usepackage{amsbsy}
\usepackage{graphicx}
\usepackage{bbm}
\usepackage{placeins }
\usepackage{color}
\usepackage{epstopdf}

\newcommand{\R}{\mathbb{R}}
\newcommand{\N}{\mathbb{N}}

\newcommand{\cv}{{\it v}}

\newcommand{\cH}{\mathcal{H}}

\newtheorem{theorem}{Theorem}[section]

\newtheorem{definition}{Definition}[section]

\newtheorem{remark}{Remark}[section]
\newtheorem{rem}{Remark}[section]
\newtheorem{example}{Example}[section]

\begin{document}

\title{\Large \bf Asymptotics for optimal controls for horizontal mean curvature flow.
}

\author{Nicolas Dirr\thanks{Cardiff School of Mathematics, Cardiff University, Cardiff, UK, e-mail: dirrnp@cardiff.ac.uk.} \and
Federica Dragoni
\thanks{Cardiff School of Mathematics, Cardiff University, Cardiff, UK, e-mail: DragoniF@cardiff.ac.uk.} \and
Raffaele Grande 
\thanks{Cardiff School of Mathematics, Cardiff University, Cardiff, UK, e-mail: GrandeR@cardiff.ac.uk}
}

\maketitle

\begin{abstract}
\noindent
The solutions to surface evolution problems like mean curvature flow can be expressed as value functions of suitable stochastic control problems, obtained as limit of a family of regularised control problems. The control-theoretical approach is particularly suited  for such problems for degenerate geometries like the Heisenberg group. In this situation a new type of singularities absent for the Euclidean mean curvature flow occurs, the so-called characteristic points. This paper investigates the asymptotic behaviour of the regularised optimal controls in the vicinity of such characteristic points.
\end{abstract}


\section{Introduction}
The evolution by  mean curvature flow is a  geometrical degenerate PDE broadly used in mathematics, see e.g. \cite{Ecker} and references therein for an overview on the subject.
Roughly speaking this describes the motion of an hypersurface contracting in the normal direction with (normal) velocity equal to  the mean curvature at that point. 
Unfortunately, even smooth surfaces  evolving by mean curvature flow
can develop singularities in finite time, 
so a weak notion  for this evolution is necessary.
The notion that we are considering here follows  a
nonlinear PDE-approach,  
based on Chen-Giga-Goto \cite{che} (see also \cite{giga})
and  
Evans-Spruck \cite{ev}. 
Roughly speaking, the idea consists in associating a 
PDE to a smooth hypersurface evolving   such 
that  the function  which solves this PDE has level 
sets which evolve by mean curvature flow. 
Then one can define the  solutions of the 
``generalized evolution by mean curvature flow'' as  
the zero-level sets of the viscosity solution of this PDE. This, so called, level set approach requires to solve (in the viscosity sense) a  degenerate parabolic PDE. 
In the last decades this evolution has been generalised to the case of sub-Riemannian geometries, i.e. to the study of hypersurfaces evolving by the so-called horizontal mean curvature flow (see \cite{{cap},{dirr}} and others). This is partially motivated by the sub-Riemannian modelling of the visual cortex applied for example to the study of image processing, developed by Citti-Sarti and al. (see e.g. \cite{{cisar},{cisar2}}).


Sub-Riemannian geometries are degenerate manifolds  where the 
Riemannian inner product is defined just on a sub-bundle of the
tangent bundle. To be more precise, we will consider $X_1,\dots,X_m$
smooth vector fields on $\R^n$  and a Riemannian inner product
defined on the distribution $\cH$ generated by such vector fields. 
Then
it is possible to define intrinsic derivatives of 
any order by taking the derivatives along the vector fields 
$X_1,\dots,X_m$. 
That allows us to write differential operators like Laplacian, 
infinite-Laplacian etc, using intrinsic derivatives. 
In particular we can write the level-set equation associated to evolution by horizontal mean curvature flow.
Even if different authors have studied this geometrical evolutions, many questions remain open due to the high degeneracy of the associated PDEs (see e.g. \cite{{cap},{dirr},{fer}}). 
To keep the computation easier and more explicit, in this paper we will focus only on the specific case of the Heisenberg group, which is the main model for a sub-Riemannian geometry. Still the approach works in the general case of Carnot-type vector fields, which in particular includes all Carnot groups.\\

A connection between   certain stochastic control problems
and a  large class of geometric evolution equations, including 
the (Euclidean) evolution by mean curvature flow,  has been found by
Buckdahn,  Cardaliaguet and Quincampoix 
in \cite{car} and  Soner and Touzi in \cite{{tou2},{tou}}.
The control, loosely speaking, constrains  the 
increments of the stochastic process
to a lower dimensional subspace of $\R^n,$ while the cost 
functional consists only of the terminal cost but
involves an essential supremum over the probability space.
It turns out that the value function solves
the level set equation associated with the geometric evolution. 
Moreover, one can show that the set of points from which 
the initial hypersurface can be reached
almost surely in a given time by choosing an appropriate 
control coincides with the set evolving by mean curvature flow. 
This stochastic approach generalizes very naturally to 
sub-Riemannian geometries by using
an intrinsic Brownian motion associated with the 
sub-Riemannian geometry. 

This approach can be used to obtain certain 
existence results in general sub-Riemannian manifolds. In particular, the value function
may be used for defining a generalized flow.
More precisely,
the value function $V$ associated to this stochastic 
control problem is 
defined as the  infimum, over the admissible controls, of the 
essential   supremum of the final cost $g$ 
(at some fixed terminal time $T>t$),   
for the controlled path $\xi^{\cv}$  starting from $x$ at the time $t$. 
We can show that 
$u(t,x ):=V(T-t,x)$ is a viscosity solution of 
the level set equation of the  evolution by horizontal 
mean curvature flow. 
So $\Gamma(t)=\{x\in \R^n\,|\, u(t,x)=0\}$ is a 
generalized evolution by horizontal 
mean curvature flow in general sub-Riemannian  manifolds.
This approach has been successfully used to study the evolution by horizontal mean curvature flow in general sub-Riemannian geometries by two of the authors, together with Max von Renesse in \cite{dirr}.
In that paper, following the approach in  \cite{car}, the authors
introduce a suitable  $p$-regularising stochastic optimal control problem, which does not degenerate when the (horizontal) gradient vanishes. The value functions $u_p$ associated to the $p$-problem do not converge as $p\to +\infty,$ but their $p-$th roots $u_p^{\frac{1}{p}}$ do, in a similar way as the $L^p$-seminorms of a measurable function converge to the essential supremum. This limit of  $u_p^{\frac{1}{p}}$  can be shown to solve in the viscosity sense the level set equation for the horizontal mean curvature flow in general sub-Riemannian geometries.

The aim of this paper is to understand better, at least at a formal level, the asymptotic behaviour of the optimal controls of these approximating control problems. 

For stochastic control problems, the optimal control is of feedback form. This means that there exist a function, depending on the value function and its derivatives, which selects the optimal control depending on the state of the system. This function is obtained by a point-wise optimization over the control space which, in the simplest case, is just the Legendre transform connecting Lagrangian and Hamiltonian, for details see e.g. \cite{Soner}

For standard control problems, it is possible to define a forward-backward system of stochastic ODEs which yield the path associated with the optimal control, see e.g \cite{Carmona}. This approach does not need the  derivatives of the value function, the price to pay is that we have a system which is forward-backward. In principle this would allow to reduce questions regarding the convergence of value functions  to convergence of a family of systems of stochastic ODEs.

Unfortunately, in the case here, the value functions doe not converge, only their p-th roots do. Re-writing everything depending on the $p$-th root of the value function would lead to a system of ODEs that still depends on the (in principle unknown) value function.
The good news is, however, that the convergence of the value function has been shown in \cite{dirr}, but without rate. Therefore our approach of studying the convergence of the Hamiltonian in $p$ nevertheless is able to shed some light on the behaviour of the approximating optimal controlled paths near characteristic points  of the limit problem.

This result gives also an idea on the structure of the optimal controls for the $p$-problem, which two of the authors are currently generalising to the case of Riemannian approximation for the horizontal mean curvature evolution, and it is crucial for showing the convergence of the stochastic approach for the Riemannian approximation to the value function solving the level-set equation in the horizantal case. These results will be contained in two follow-up papers in preparation.

The paper is organised as follows:\\
 In Section 2 we will give some definitions related to sub-Riemmanian geometries.
 In the Section 3 we will recall some known results about horizontal mean curvature flow, while in  Section 4 we focus on recalling the stochastic approach from \cite{dirr}.
 In Section 5  introduce  the $p$-Hamiltonian for a generic Carnot-type geometry.
 In the Section 6 we will obtain the structure of  the optimal control for the $p$-Hamiltonian, for large $p$ in the case of the Heisenberg group.
 Finally in Section 7 we will look at the behaviour of the optimal control for large $p$ near the characteristic points and give some numerical examples.

\section{Preliminary}

In order to prove our main result, we have to recall briefly the basic definitions about sub-Riemmanian geometries. For further remarks and definitions on this topic we refer to \cite{mont} and \cite{BLU} for the case of Carnot groups.\\

Let $M$ be a $n$-dimensional smooth manifold, we recall that  a \emph{distribution} is a subbundle of the tangent bundle, i.e. as vector space
\begin{equation*}
\mathcal{H} := \{ (x,v) | x \in M \ \ v \in \mathcal{H}_{x} \},
\end{equation*}
where   $\mathcal{H}_{x}$ is a subspace of  the tangent space $T_{x} M$ at  every point $x\in M$.\\
Given two vector fields  $X$,$Y$  defined on a  manifold $M$, we can consider   the bracket between $X$ and $Y$, that is the vector field 
acting on the  smooth functions $f:M \rightarrow \mathbb{R}$ as
\begin{equation*}
[X,Y](f)= XY(f) - YX(f).
\end{equation*}
Let us now consider  a family of vector fields $\mathcal{X} := \{ X_{1}, \dots , X_{m} \}$, we define the set of all the $k$-brackets of $\mathcal{X} $ as
\begin{equation*}
\mathcal{L}^{(k)} (\mathcal{X}) := \{ [X,Y] | X \in \mathcal{L}^{(k-1)}(\mathcal{X}), \ \ Y \in \mathcal{L}^{(1)}(\mathcal{X}) \},
\end{equation*}
with $\mathcal{L}^{(1)} (\mathcal{X})= \mathcal{X}$.
The associated Lie algebra is the set of all brakets between the vector fields of the family
\begin{equation*}
\mathcal{L}(\mathcal{X}) := \{ [X_{i} , X_{j}^{(k)} ] | X^{(k)}_{j} \in \mathcal{L}^{(k)} (\mathcal{X}), \ k \in \mathbb{N} \}.
\end{equation*}

We can now recall the  H\"{o}rmander condition.
\begin{definition} [H\"{o}rmander condition]
Let $M$ be a smooth manifold and $\mathcal{H}$ a distribution defined on $M$ and let  $\mathcal{X}$ be a family fo vector fields spanning $\mathcal{H}$. We say that the distribution is \emph{bracket generating}   if and only if, at any point, the Lie algebra $\mathcal{L}(\mathcal{X})$ spans the whole tangent space at that point.
Moreover we say that the family of vector fields $\mathcal{X}$ satisfy the \emph{H\"{o}rmander condition} if and only if there exists $r\in \N$ such that $T_xM=\bigcup_{k=1}^r\textrm{Span}\big(\mathcal{L}^{(k)} (\mathcal{X})(x)
\big)$; in this case the natural number $r$ is called step of the group.
The space $\mathcal{H}_x=\textrm{Span}\big( \mathcal{X}_1, \dots, \mathcal{X}_m\big)(x)$ is usually called {\em horizontal space} at the point $x$.
\end{definition}

In this setting not all the curves on the manifolds will be admissible.

\begin{definition}
Let $M$ be a smooth manifold and  $\mathcal{H}$ a bracket generating distribution defined on $M$ and  generating by a family of  vector fields  $\mathcal{X} = \{ X_{1}, \dots , X_{m} \}$.
Consider  an absolutely continuous curve  $\gamma:[0,T] \rightarrow M$, we say that $\gamma$ is a \emph{horizontal curve} if and only if
\begin{equation*}
\dot{\gamma}(t) \in \mathcal{H}_{\gamma(t)}, \ \ \mbox{for a.e} \ t \in [0,T]
\end{equation*}
or, equivalently, if there exists a measurable function   $h : [0,T] \rightarrow \mathbb{R}^{N}$ such that
\begin{equation*}
\dot{\gamma}(t) = \sum_{i=1}^{m} h_{i}(t) X_{i}(\gamma(t)), \ \ \mbox{for a.e} \ t \in [0,T],
\end{equation*}
where $h(t)=(h_{1}(t), \dots  , h_{m}(t))$.
\end{definition}


In the next example we will introduce the most significant and famous model in this setting: the Heisenberg group.

\begin{example}[The Heisenberg group] \label{a2}
 For a formal definition of the Heisenberg group and the connection between its structure as non commutative Lie group and its manifold structure we refer to \cite{BLU}.
Here we simply introduce the 1-dimensional  Heisenberg group as the geometries induced  on  $\mathbb{R}^{3}$ 
by the vector fields
\begin{equation*}
X_1(x)= \begin{pmatrix}
1\\ 0 \\ - \frac{x_{2}}{2}
\end{pmatrix}
\quad  \textrm{and} \quad
X_2=
\begin{pmatrix}
0\\1 \\ \frac{x_{1}}{2}
\end{pmatrix},
\quad \forall \ x=(x_1,x_2,x_3)\in \mathbb{R}^3.
\end{equation*} 
The horizontal space in this case is given by $H_x=\textrm{Span}\big(X_1(x),X_2(x)\big)$; so in particular at the origin the horizontal space is the plane $x_3=0$.\\
Note that the  above vector fields satisfy the H\"ormander condition with step 2: in fact $[X_1,X_2](x)=\begin{pmatrix}0\\0\\1\end{pmatrix}$
for any $x\in \mathbb{R}^3$.

\end{example}

From now on we consider only the case where the starting topological manifold $M$ is the Euclidean $\mathbb{R}^N$. \\

For later use we also introduce the matrix associated to the vector fields $X_{1} , \dots , X_{m}$, which is the
$N\times m$ matrix defined as
\begin{equation*}
\sigma(x)=[X_{1}(x), \dots , X_{m}(x)]^{T}.
\end{equation*}
\begin{example}
In the case of the Heisenberg group  introduced in the Example \ref{a2}, the matrix $\sigma$ is given by
\begin{equation} \label{hei}
\sigma(x)= \begin{bmatrix} 1 & 0 & - \frac{x_{2}}{2} \\ 0 & 1 & \frac{x_{1}}{2} \end{bmatrix},
\quad
\forall \, x=(x_{1},x_{2},x_{3}) \in \mathbb{R}^{3}.
\end{equation}
\end{example}
Moreover, in this paper we will concentrate on a sub-Riemannian geometries with a particular structure: the so called Carnot-type geometries.
In general, for Carnot-type geometries, the matrix $\sigma$ assumes the following structure:
\begin{equation*}
\sigma(x) = \begin{bmatrix} I_{m \times m} &  A(x_{1}, \dots x_{m}) \end{bmatrix},
\end{equation*}
where the matrix $A(x_{1}, \dots , x_{m})$ is a $(N-m) \times m$  depending only on the first $m$ components of $x$.\\
All Carnot groups are Carnot-type geometries (see e.g. \cite{BLU} for definitions and examples of Carnot groups).
The previous assumption on the structure of the vector fields will allow us to consider an easy and explicit  form fort the Riemannian approximation. Nevertheless the approach introduced apply also to the case where this additional structure is not fulfilled.

\section{Horizontal mean curvature flow}
Given a smooth hypersurface $\Gamma$ on $\R^N$ (or more in general on a $N$-dimensional manifold $M$), we indicate by $n_E(x)$ the standard (Euclidean) normal to $\Gamma$ at the point $x$. We now consider the vector fields   $X_{1}, \dots , X_{m}$ introduced in the previous section (i.e. spanning a bracket-generating distribution $\cH_x$), 
and we look at the unit vector obtained projecting the Euclidean normal on the distribution $\cH_x$ generated by $X_{1}, \dots , X_{m}$, see the following definition.
\begin{definition}[Horizontal normal] \label{c6d10}
Given a smooth hypersurface $\Gamma$ on $\R^N$ and a family of vector fields $X_{1}, \dots , X_{m}$  satisfying the H\"ormander condition, the \emph{horizontal normal} $n_{0}(x)$ is the renormalized projection of the Euclidean normal $n_E(x)$ on 
the horizontal space $\cH_x=\textrm{Span}\big(X_1(x),\dots ,X_m(x)\big)$.
\end{definition}
Since $n_{0}(x)\in \textrm{Span}\big(X_1(x),\dots ,X_m(x)\big)$ and the Riemannian inner product is introduced in such a way  $X_{1}, \dots , X_{m}$ are orthonormal, then, whenever the projection of $n_E(x)$ onto $\cH_x$ does not vanish,  there exists $h_1,\dots,h_m$ measurable functions such that
\begin{equation*}
n_{0}(x)= \frac{ h_1(x)X_1(x)+\dots +h_m(x)X_m(x)}{\sqrt{h^2_1(x)+\dots +h^2_m(x)}}\in \cH_x\subset \R^N , \ \ x \in \Gamma.
\end{equation*}

With an abuse of notation we sometimes identify $n_{0}$ with the associated $m$-valued vector 
\begin{equation}
\label{normal_coordinateVector}
n_{0}(x) \rightarrow 
\left(\frac{ h_1(x)}{\sqrt{h^2_1(x)+\dots +h^2_m(x)}},\dots, \frac{ h_m(x)}{\sqrt{h^2_1(x)+\dots +h^2_m(x)}}
\right)^T\in \R^m.
\end{equation}
The main difference between the standard  normal and the  horizontal normal is that the second  may not exist even for smooth hypersurfaces. In fact whenever the Euclidean normal is orthogonal to the horizontal plane $\cH_x$, then the horizontal normal cannot be introduced. The points where this  happens are called characteristic points, see the definition below.
\begin{definition}[Characteristic points] \label{CharacteristicPoints}
 Given a smooth hypersurface $\Gamma$ on $\R^N$, the characteristic points  occur whenever $n_E(x)$ is orthogonal to the horizontal plane $\mathcal{H}_x$, then its projection on such a subspace vanishes, i.e.
 $$
h^2_1(x)+\dots +h^2_m(x)=0.
 $$
\end{definition}

We recall that for every smooth hypersurface the \emph{mean curvature} at a point  is defined as the divergence of the Euclidean normal at that point.
Similarly, for every smooth hypersurface, we can now introduce the  horizontal mean curvature.
\begin{definition}
[Horizontal mean curvature]
\label{c6gl}
Given a smooth hypersurface $\Gamma$ and a non characteristic point $x\in \Gamma$, the 
 \emph{horizontal mean curvature} is defined as the horizontal divergence of the horizontal normal, i.e.
$
k_{0}(x) = div_{\mathcal{H}} n_{0}(x),
$
where $n_0(x) $ is the m-valued vector associated to the horizontal normal defined in \eqref{normal_coordinateVector}, while $div_{\mathcal{H}} $ is the divergence w.r.t. the vector fields $X_1,\dots, X_m$, i.e.
\begin{equation}
\label{horizontalMCF}
k_{0}(x) = X_1\left(\frac{h_1(x)}{\sqrt{\sum_{i=1}^mh_i^2(x)}}\right)+
\dots+
X_m\left(\frac{h_m(x)}{\sqrt{\sum_{i=1}^mh_i^2(x)}}\right) \ \ x \in \partial \Gamma .
\end{equation}
\end{definition}
Obviously the horizontal mean curvature is never defined at characteristic points, since there the  horizontal normal does not exist.\\

We can finally introduce the main definition of this section.

\begin{definition}[Evolution by mean curvature flow] \label{c6glo}
Let $\Gamma_t$ be a family of smooth hypersurfaces in $\R^N$, depending on a time parameter $t\geq 0$. 
We say that $\Gamma_{t}$ is an \emph{evolution by horizontal mean curvature flow} 
of some hypersurface $\Gamma$ if and only if $\Gamma_0=\Gamma$ and for any smooth horizontal  curve $\gamma: [0,T] \rightarrow \mathbb{R}^{N}$ such that $\gamma(t) \in \Gamma_{t}$ for all $t \in [0,T]$, 
the horizontal normal velocity $v_{0}$ is equal to minus the horizontal mean curvature, i.e.
\begin{equation} \label{sla}
v_{0}(\gamma(t)): = - k_{0}(\gamma(t)) n_0(\gamma(t)), \ \ t \in [0,T],
\end{equation}
where $n_{0}$ and $k_{0}$ as respectively the horizontal normal and the horizontal mean curvature introduced in Definitions \ref{c6d10} and \ref{c6gl}.
\end{definition}
We recall  that \eqref{sla} is not defined at characteristic points, then the need to develop a generalised notion of evolution by mean curvature flow which can deal with characteristic points and general singularities, as we will do in the next section following some very well-known approaches, already used to deal with singularities  in the Euclidean case.

\section{The level-set equation: a stochastic approach.} 
In this section we introduce the  level set equation for the (generalised) evolution by horizontal mean curvature flow, and a stochastic representation for the viscosity solutions of that equation.\\
 The level set approach for the Euclidean evolution was introduced by Evans and Spruck  in \cite{ev} and Chen, Giga and Goto in \cite{che}.
We briefly recall this approach directly for  the horizontal evolution introduced in \eqref{sla}, for more details see \cite{dirr} and \cite{cap}.
The basic  idea starts by parametrising all (smooth) hypersurface involved as zero level sets, i.e.
$$
\Gamma=\Gamma_0= \big\{x \in \mathbb{R}^{N}| u_0(x)= 0 \big\}
\quad
\textrm{and}
\quad
\Gamma_t= \big\{x \in \mathbb{R}^{N}| u(t,x)= 0 \big\},
$$ 
for some smooth function $u:[0,+\infty)\times \R^N\to \R$. From now on, we indicate the points $x\in \Gamma_t$ as $x(t)$. Then 
 the Euclidean normal is simply 
$n_E(x(t))= \frac{\nabla u(t,x(t))}{| \nabla u(t,x(t)) |}$, where the gradient is done only w.r.t. the space variable $x$, for all $x(t)\in \Gamma_t$. This implies that the horizontal normal (at non-characteristic points) can be expressed as
\begin{equation} \label{c6wi}
n_{0}(x(t))= \left( \frac{X_{1} u (t,x(t))}{\sqrt{\sum_{i=1}^{m} (X_{i} u(t,x(t)))^2}} , \dots , \frac{X_{m} u (t,x(t))}{\sqrt{\sum_{i=1}^{m} (X_{i} u(t,x(t)))^2}} \right) \in \mathbb{R}^{m},
\end{equation}
(where above we have simply identifies the normal with its coordinate vector in $\R^m$).
Note that $ D_{\mathcal{X}} u =(X_1u,\dots , X_mu)\in \R^m$ is the so called horizontal gradient, then $n_{0} = \frac{D_{\mathcal{X}} u}{|D_{\mathcal{X}} u|}$ and by $| \cdot|$ we indicate the standard Euclidean norm in $\mathbb{R}^{m}$. 
Similarly we can then write the horizontal mean curvature given in \eqref{horizontalMCF} as 

\begin{equation} \label{c6cru}
k_{0}(x(t))= \sum_{i=1}^{m} X_{i} \left( \frac{X_{i}u (t,x(x))}{\sqrt{\sum_{i=1}^{m} (X_{i} u(t,x(t)))^2}} \right) .
\end{equation}

Applying \eqref{c6wi} and \eqref{c6cru} to the Definition \ref{c6glo}, we obtain that, whenever $\Gamma_t$ satisfies the evolution \eqref{sla} and $|D_{\mathcal{X}} u|>0,$  then $u$ solves the following PDE:
\begin{equation} \label{horMCF}
u_{t} = Tr( (D^{2}_{\mathcal{X}} u)^{*}) - \bigg< (D^{2}_{\mathcal{X}} u)^{*} \frac{D_{\mathcal{X}} u}{|D_{\mathcal{X}} u|} , \frac{D_{\mathcal{X}} u}{|D_{\mathcal{X}} u|} \bigg>,
\end{equation}
where $(D^{2}_{\mathcal{X}}u)^{*}$  is the symmetric horizontal Hessian, that is
\begin{equation*}
((D^{2}_{\mathcal{X}} u)^{*})_{ij}:= \frac{X_{i}(X_{j}u) + X_{j}(X_{i}u)}{2}.
\end{equation*}
Equation \eqref{horMCF} was introduced and studied in  \cite{cap} and \cite{ev}. Different approaches lead to different ways t interpret the singularity $|D_{\mathcal{X}} u|,$ which happens even if the surface is smooth in the Euclidean sense. To see this, consider the unit sphere in three dimensions centred at $(0,0,1).$ In the Heisenberg geometry, the horizontal gradient to any level set function vanishes in the point $(0,0,1).$ This example will be used for some numerical illustrations in the final chapter.

Equation \eqref{horMCF} is very degenerate and in general the solutions will need to be interpreted in the sense of the viscosity solutions (see \cite{lio}, for a definition and properties).

Uniqueness of viscosity solutions is in full generality an open problem due to the presence of points where 
$|D_{\mathcal{X}} u|$ vnishes, i.e. characteristic points.

Here we concentrate on the stochastic approach  initiated independently  by Cardaliaguet, Buckdahn and Quincampoix in \cite{car} and by Soner and Touzi in \cite{{tou2},{tou3}}.
The same approach was later generalised by Dirr, Dragoni and Von Renesse in \cite{dirr}   to cover the horizontal case considered in this paper.
Roughly speaking the idea  consists in expressing the viscosity solution of the level set equation as value function of suitable associated stochastic controlled systems. This is made more precise in the following result.
\begin{theorem}[\cite{dirr}]
\label{theoremMainDDR}
Let $(\Omega,\mathcal{F},\{\mathcal{F}_{t} \}_{t \geq 0},\mathbb{P})$ be a filtered probability space and $B$ is a $m$-dimensional Browinian motion adapted to the filtration $\{\mathcal{F}_{t} \}_{t \geq 0}$.

Let $g: \mathbb{R}^{N} \rightarrow \mathbb{R}$ be a bounded and H\"older  function. Let us consider $T>0$. For any $(t,x)  \in [0,T] \times {\mathbb R}^{N}$, we define
\begin{equation}\label{c6e200}
V (t,x) = \inf _{v \in \mathcal{A}}  ess \sup_{\Omega} g(\xi^{t,x, \nu}(T)) ,
\end{equation}
 where
\begin{equation}
\mathcal{A} = \{ \nu \in Sym(m) | \ \nu \geq 0 \ \  I_{m} - \nu^{2} \geq 0, \ \ Tr(I_{m} - \nu^{2})=1 \},
\end{equation}
and $\xi^{t,x, \nu}$ are the solution of the stochastic controlled dynamics
\begin{equation}
\begin{cases} d \xi^{t,x, \nu}(s) = \sqrt{2} \sigma^{T} ( \xi^{t,x, \nu} (s)) \circ dB^{\nu}(s), \  \  \ \ \ s \in (t,T],  \\ dB^{\nu}(s) = \nu(s) dB_{m}(s), \ \ \ \  \ \   \ \ \ \ \ \ \ \ \ \  \ \ \ \ \ \ \    s \in (t,T], \\ \xi^{t,x, \nu}(t) = x. \end{cases}
\end{equation}
Let  also us assume that the matrix $\sigma$ is $m \times N$ H\"ormander matrix with smooth coefficients and that $\sigma$ and $\sum_{i=1}^{m} \nabla_{X_{i}} X_{j}$ are Lipschitz functions.
Then the value function $V$ defined in \eqref{c6e200} is a viscosity solution of the level set equation  \eqref{horMCF}. Here $\circ$ denote the Stratonovich differential.
\end{theorem}
The approach to prove the theorem above is classical and follows the ideas introduce in \cite{car} for the standard (Euclidean) evolution, which means that require to first consider a more regular problem known as $p$-regularizing problem, which we introduce in the next section.

At characteristic points, (i.e. $|D_{\mathcal{X}} u|=0,$) the approach in \cite{dirr} yields a discontinuous nonlinearity (different for sub-and supersolutions): for subsolutions we get that $u_t$ equals the minimal eigenvalue of the horizontal Hessian, while  for supersolutions we get the maximal eigenvalue. The corresponding control would be a projection on the respective eigenspace. One result here (see Remark \ref{Rem64}) is a refinement of this conclusion: the optimal control is not unique if both eigenvalues are equal, otherwise it is the projection on the eigenspace of the {\em maximal } eigenvalue.

\section{The $p$-regularizing problem.}
To show directly that the  value function $V$ defined in \eqref{c6e200} is a viscosity solution on Equation  \eqref{horMCF} is extremely hard, the two main difficulties being that the PDE is highly degenerate and the value function is a $L^{\infty}$-norm. The idea from  \cite{car}  is to consider the value function associated to $L^p$-norm  approximating  (at leats on bounded sets)\eqref{c6e200}, which we indicate by $V_p$, and show that this new value function solves in the viscosity sense the  corresponding PDE. Then we can recover the result given in Theorem \ref{theoremMainDDR}
by a  limit-argument as $p\to +\infty$.
Note that   the $p$-problem associated to the new value function $V_p$ is far more regular than the level set equation  \eqref{horMCF} i
 and in fact the associated $p$-Hamiltonian  has no points of discontinuity.
For all $1 < p < + \infty$ we define the $p$-value function as
\begin{equation}
V_{p}(t,x):= \inf_{\nu \in \mathcal{A}} \mathbb{E}[ g^{p}(\xi^{t,x, \nu}(T))]^{\frac{1}{p}}.
\end{equation}
where $\xi^{t,x, \nu}$ and $\mathcal{A}$ are defined as in  Theorem \ref{theoremMainDDR}.
Then one can show (see \cite{dirr}) that $V_{p}$ solves in the viscosity sense 
\begin{equation}
\begin{cases} -(V_{p})_{t} + H_{p}(V_{p}, DV_{p}, D^{2}V_{p}) = 0,  \ \  x \in \mathbb{R}^{n}, \ t \in [0,T), \\ V_{p}(T, x) = g(x), \ \ \ \ \ \ \ \ \ \ \ \ \ \ \ \ \ \ \ \ \ \ \ \   x \in \mathbb{R}^{n}, \end{cases}
\end{equation} 
where the $p$-Hamiltonian $H_{p}$ is defined as
\begin{equation} \label{hami21}
H_{p}(r,q, M):= \sup_{\nu \in \mathcal{A}} \bigg[-(p-1) r^{-1} Tr[\nu \nu^{T} q q^{T}] + Tr[\nu \nu^{T} M]\bigg],
\end{equation}
and $q \in \mathbb{R}^{m}$ and $M \in Sym(m)$ where $\sigma$ is the matrix in \eqref{hei}. \\ 

The aim of this article is to find information on the structure of the optimal control for the $p$-Hamiltonian associated to the $p$-Hamilton-Jacobi equation regularising the the level set equation for the evolution by horizontal mean curvature flow.\\
For sake of simplicity, we now introduce the following function
\begin{equation}\label{littleh}
h_{p}(r,q,M, \nu):= -(p-1)r^{-1} Tr[\nu \nu^{T} q q^{T}] + Tr[\nu \nu^{T} M],
\end{equation}
so that the $p$-Hamiltonian can be rewritten as
\begin{equation} \label{gener}
H_{p}(r,q, M) = \sup_{\nu \in \mathcal{A}} h_{p}(r,q,M, \nu).
\end{equation}

As proved in \cite{tou2}, finding the infimum of the optimal controls for the solution $V_{p}$ is equivalent to optimise the supremum of the Hamiltonian as defined in \eqref{hami21}.
Then we will concentrate in finding the structure of the optimal controls giving the supremum in 
\eqref{gener}, for $p$ large enough.\\

We conclude this section with the following two remarks, which will be very useful for the later results.
\begin{remark}
It is possible (see \cite{car}) to rewrite the set of admissible controls as
\begin{equation} \label{mass}
\mathcal{A}= Co \{  \nu = I_{m} - a \otimes a, |\, a\in \R^m, \ \ |a|=1 \}.
\end{equation}
\end{remark}

\begin{remark}
Note that it is possible to consider $r > 0.$ As we are not interested in the value function itself, but only in its level sets, we can use as initial datum a positive function, e.g. $u_0(x)=1+\tanh ({\rm dist}(x,\Sigma_0)).$
By the comparison principle, this will remain positive. Note that the level of interest here will be the 1-level set, not the zero level set, which is empty. For comparison principles for a large set of hypersurfaces in the Heisenberg group we refer to \cite{fer}.
\end{remark}

\section{Optimal control for the $p$-Hamiltonian in $\mathbb{H}^{1}$}

In this section we will prove the main results of the paper. We will focus on the optimal control for the $p$-Hamiltonian defined in \eqref{hami21} and, in order to keep the computations easier, we will consider only  the case of the 1-dimensional  Heisenberg group, introduced in Example \ref{a2}.
We will divide our investigation in two separate cases: the case $q \neq 0$ and the case $q=0$. Remember that the case $q=0$ corresponds to the case when the horizontal gradient vanishes, which is associated to the characteristic points introduced in Definition \ref{CharacteristicPoints}.
Let us recall that, in the case of the 1-dimensional Heisenberg group, $N=3$ and $m=2$. \\

Next we introduce  the main idea:
note that it is possible to express any admissible control as $\nu = I_{2} - a \otimes a$ with $|a|=1$ (see \eqref{mass}).  Moreover a generic unit vector $a$ can be expressed by rotating any given fixed direction. Therefore, fixed an initial direction, to maximise the supremum  in \eqref{hami21}
on the set of all admissible controls  $ \mathcal{A}$
 can be reduced to maximise the same function $h_p$ among the rotational angle  $\theta\in [0, 2 \pi)$, which will be  much easier. Since this can be done starting from any fixed direction, we will rotate exactly the direction which we know to be associated to the optimal problem for the limit case $p=+\infty$, that is $\frac{q}{|q|}$, (remember  $q$ is the gradient variable). For more details on the optimal control in the case $p=+\infty$ we refer to \cite{dirr}. This of course cannot be done whenever $q=0$, then we will treat that case later, rotating a different starting direction.

 \subsection{Case $q \neq 0$: non-characteristic points.}
 Now let us fix the variables $r$,$q$,$M$ with $q\neq 0$. 
For sake of simplicity we also assume that $M$ is a diagonal matrix, i.e. there exist $\lambda_1,\lambda_2\in \R$ such that
\begin{equation} \label{bre1}
M= \begin{bmatrix} \lambda_{1} & 0 \\ 0 & \lambda_{2} \end{bmatrix}.
\end{equation}
In the next remark we highlight  that the assumption above on  $M$ is indeed not restrictive.
\begin{remark}
\label{diagonalMatrix}
This optimisation is taken at a fixed point $x$ and depending on variable $r,q,$ and $M.$ When writing these quantities explicitly, i.e. by specifying their entries, we implicitly refer to a coordinate system on $\R^2.$ We may choose  an orthonormal coordinate system at this point. For simplicity we choose it in such a way that the symmetric matrix $M$ is diagonal.

Note that for $h_p$ as in \eqref{littleh} it holds that
$$h_p(r,q,M,\nu)=h_p(r, {\mathcal O}q, {\mathcal O}M{\mathcal O}^T, {\mathcal O}\nu)$$ 
for any orthonormal matrix ${\mathcal O},$ i.e.
such that ${\mathcal O}{\mathcal O}^T=I:$ 

For two vectors $v_1$ and $v_2$ we have
$${\mathcal O}v_1({\mathcal O}v_2)^T={\mathcal O}(v_1v_2^T){\mathcal O}^T,
$$ and for two matrices $A$ and $B$ we have 
$$
Tr\big(({\mathcal O}A{\mathcal O}^T)(\mathcal {O}B{\mathcal O}^T)\big)=
 Tr({\mathcal O}(AB){\mathcal O}^T)=
 Tr ({\mathcal O}^T{\mathcal O} (AB))=Tr(AB).
$$

\end{remark}

We now introduce the following unit vector in polar coordinates:
\begin{equation} \label{ja}
n_{\infty}:= \frac{q}{|q|} = \begin{bmatrix} \cos \alpha \\ \sin \alpha \end{bmatrix},
\end{equation}
for a suitable associated angle $\alpha \in [0, 2 \pi)$ fixed.

 We now express any admissible control $\nu\in \mathcal{A}$ as
  \begin{equation}
  \label{Max}
  \nu=\nu_{\theta} \in \mathcal{A} \Longleftrightarrow \nu_{\theta} = I_{2} - n_{\theta} \otimes n_{\theta},
 \end{equation}
 where
   \begin{equation}
   \label{www21}
\begin{aligned} 
n_{\theta} := R_{\theta} n_{\infty} = \begin{bmatrix} \cos \theta & - \sin \theta \\ \sin \theta & \cos \theta \end{bmatrix} \begin{bmatrix} \cos \alpha \\ \sin \alpha \end{bmatrix}  = \begin{bmatrix} \cos (\theta + \alpha) \\ \sin( \theta + \alpha) \end{bmatrix}, \ \  \mbox{for $ \theta \in[0, 2 \pi)$}.
\end{aligned}
 \end{equation}
Using \eqref{Max}, we can rewrite
 all the admissible control $\nu\in \mathcal{A}$ as
\begin{equation} \label{ca2}
\nu_{\theta} := I_{2} -  n_{\theta} \otimes n_{\theta} = \begin{bmatrix} \sin^{2}(\theta + \alpha) & - \sin(\theta + \alpha) \cos (\theta + \alpha) \\ - \sin(\theta + \alpha) \cos (\theta + \alpha) &   \cos^{2}(\theta + \alpha)  \end{bmatrix}.
\end{equation}
Moreover $\nu_{\theta}=\nu_{\theta+\pi}$, then we can restrict our attention to $ \theta \in[0,  \pi)$.\\
It is also easy to check that $\nu_{\theta}$ is indeed an admissible control-matrix, in fact it is a symmetric projection matrix, in fact: $\nu_{\theta}=\nu^{T}_{\theta}$, which trivially  implies $\nu_{\theta} \nu^{T}_{\theta} = \nu^{2}_{\theta}$. 
Moreover for any generic vector $q \in \mathbb{R}^{2}$, we have
\begin{equation*} 
\nu^{2}_{\theta} q= \nu_{\theta} (q - n_{\theta} < n_{\theta}, q >) = \nu_{\theta} q - <n_{\theta} , q> ( n_{\theta} - n_{\theta} <n_{\theta} , n_{\theta}>)   = \nu_{\theta} q.
\end{equation*} 
Given  a diagonal matrix $M$ as in \eqref{bre1}, and using that $\nu_{\theta}$ is symmetric and a projection matrix,
 then
 \begin{equation} 
  \label{da}
Tr(\nu_{\theta} \nu^{T}_{\theta} M)= Tr( \nu^{2}_{\theta} M) = Tr( \nu_{\theta} M) =  \lambda_{1} (\nu_{\theta})_{11}+\lambda_{2} (\nu_{\theta})_{22},
\end{equation} 
where by $(\nu_{\theta})_{ij}$ we indicate the coefficient in position $(i,j)$ of the matrix $\nu_{\theta}$.\\

The next  remarks will be useful for  the later proofs of the main results.

\begin{remark}  
Given any admissible control $\nu_{\theta}$, expressed by   \eqref{ca2},  for every given vector $q \in \mathbb{R}^{2}$, we have
\begin{equation*}
\nu_{\theta} q = q - n_{\theta} < n_{\theta} , q>.
\end{equation*}
\end{remark}

\begin{remark}
For all $q \in \mathbb{R}^{2}$, we deduce
\begin{equation}\label{luc}
Tr( \nu_{\theta} \nu^{T}_{\theta} q q^{T}) = Tr(\nu_{\theta} q \otimes \nu_{\theta} q)  = | \nu_{\theta} q|^{2}.
\end{equation}
\end{remark}
\begin{remark} \label{vi}
Given any $q \in \mathbb{R}^{2}$ and by using  \eqref{luc}, we have
\begin{align*}
|\nu_{\theta}q|^{2} \! = \! |q|^{2} \! - \! 2 \! <q, n_{\theta}> <q, n_{\theta}> \! + \! <q, n_{\theta}>^{2} = |q|^{2} - <q, n_{\theta}>^{2}.
\end{align*}
\end{remark}
We can now find the structure of the optimal controls for large $p$.

\begin{theorem} \label{tub}
Let us consider the $p$-Hamiltonian $H_p$ 
 introduced 
in \eqref{hami21}.
 Fixed $r$, $q$  and $M$ and assume that  $q \neq 0$ and that $M$ is a diagonal matrix as in \eqref{bre1}.
 Then, for large $p$, the optimal control  is $\overline{\nu}=\nu_{\overline{\theta}}$,  where  $\nu_{\theta}$ is defined in \eqref{ca2} and
\begin{equation}
\overline{\theta} = \overline{C}\; \frac{1}{ p} + O \bigg( \frac{1}{p^{2}} \bigg), 
\end{equation}
  with $\overline{C}=
\frac{C_{1}}{C_{2} }$, $C_{1} = (\lambda_{1} - \lambda_{2}) \sin 2 \alpha$,
$\alpha$ is defined in \eqref{ja},
 and $C_{2}= 2 r^{-1} |q|^{2}>0$.
\end{theorem}
\begin{proof}

By \eqref{da} and Remark \ref{vi} we can rewrite the function $h_p$ introduced in \eqref{littleh} as 
 \begin{equation}
 \label{rowena}
h_{p}(r,q, M , \nu_{\theta})= -(p-1) r^{-1} (|q|^{2} - <q, n_{\theta}>^{2}) +  \lambda_{1} (\nu_{\theta})_{11} + \lambda_{2} (\nu_{\theta})_{22} ,
\end{equation}
where $\nu_{\theta}$ is any admissible control expressed as in  \eqref{ca2}. 
We observe that the first term in \eqref{rowena} can be written as
\begin{equation*}
-(p-1) r^{-1}(|q|^{2} - <n_{\theta} , q>^{2}) =- (p-1) r^{-1} |q|^{2} (1 - <n_{\theta} , n_{\infty} >^{2}),
\end{equation*}
where $n_{\infty}$ is the vector introduced in \eqref{ja} (remember that $n_{\infty}$  depends on the fixed vector $q$). Then $h_p$ becomes
\begin{align*} 
h_{p}(x,r,q,M, \nu_{\theta}) = (p-1) r^{-1} |q|^{2} (1-<n_{\theta} , n_{\infty}>^{2}) + \lambda_{1}\sin^{2} (\alpha + \theta) + \lambda_{2}\cos^{2}(\alpha + \theta).
\end{align*}
Recalling the definitions of $n_{\infty}$ and $n_{\theta}$, given respectively in  \eqref{ja} and  in \eqref{www21}, we compute 
\begin{align*}
<n_{\theta} , n_{\infty}>^{2} &=  ( \cos (\alpha + \theta) \cos \alpha + \sin (\alpha + \theta) \sin \alpha)^{2}  \\ &=( \cos^{2} \alpha \cos \theta - \sin \alpha \cos \alpha \sin \theta + \sin \alpha \cos \alpha \sin \theta + \sin^{2} \alpha \cos \theta)^{2} \nonumber  \\ &= (\sin^{2} \alpha + \cos^{2} \alpha)^{2} \cos^{2} \theta   = \cos^{2} \theta.  
\end{align*}
Thus $1-<n_{\theta} , n_{\infty}>^{2}=1-\cos^{2} \theta=\sin^{2} \theta$,
and the function $h_p$ simplify as below:
\begin{equation}
\label{LargeP}
h_{p}(x,r,q,M, \nu_{\theta}) =  - (p-1) r^{-1} |q|^{2} \sin^{2} \theta + \lambda_{1} \sin^{2}(\theta + \alpha) + \lambda_{2} \cos^{2}(\theta + \alpha) .
\end{equation}
We need to find the supremum of $h_p$ among all the admissible controls, i.e. among all $\theta\in[0,\pi)$,
hence we look at the stationary points. \\
For sake of simplicity, fixed $x,r,q,M$, we introduce the following notation:
$$
f_p(\theta):=h_{p}(x,r,q,M, \nu_{\theta}).$$

Taking the derivative, we find
\begin{equation*}
\begin{aligned}
f'_{p}(\theta) =&- 2(p-1) r^{-1} |q|^{2} \sin \theta \cos \theta  +2(\lambda_{1} - \lambda_{2}) \cos (\alpha + \theta) \sin (\alpha + \theta)
\\ 
 =& -(p-1) r^{-1} |q|^{2} \sin 2 \theta + (\lambda_{1} - \lambda_{2}) \sin(2 \theta + 2 \alpha). 
\end{aligned}
\end{equation*}
Note that for large $p$ the stationary points occur for $\theta$ near 0 or $\theta$ near $\frac{\pi}{2}$. Looking at the function $
f_p$  we can see that for $\theta\approx 0$ we get the maximum while for $\theta\approx \frac{\pi}{2}$ we select the minimum.
To find the zero of  $
f'_p$
 we introduce a suitable linearisation for the derivative function $f_{p}'$.
 The previous remark suggests us the following 
 ansatz:
 $$\theta= \frac{\beta}{p}.$$
We  are now going to use the Taylor expansion of $\cos   \frac{\beta}{p} $ and $\sin  \frac{\beta}{p}$, as $p\to+\infty$,
$$
\cos \bigg( \frac{\beta}{p} \bigg) = 1 + O \bigg( \frac{1}{p^{2}}\bigg)
\quad
\textrm{and}\quad  
\sin \bigg(\frac{\beta}{p} \bigg) = \theta + O \bigg( \frac{1}{p^{3}} \bigg).
$$
Using the ansatz $\theta= \frac{\beta}{p} $ and the above Taylor's expansions,  $f_{p}'$ can be rewritten as
\begin{align*} 
 f'_{p}(\theta) =& -(p-1) r^{-1} |q|^{2} \left( \theta + O\left( {1}/{p^{3}} \right) \right)  
  \\ &+ (\lambda_{1} - \lambda_{2})
   \left(\sin 2 \alpha  \left( 1 + O\left( 
 {1}/{p^{2}} \right) 
   \right) + \cos 2 \alpha \left( \theta + O\left( {1}/{p^{3}} \right) \right) \right) \\ 
   =& - C (p-1) \theta + (\lambda_{1} - \lambda_{2}) \sin 2 \alpha +  +  \theta \cos 2 \alpha + O \bigg(\frac{1}{p^{2}} \bigg) ,\end{align*}
where $C=  r^{-1} |q|^{2} $.
Then, for $p$ large,  $f_{p}'(\theta_p)= 0$ if and only if
\begin{equation*}
\theta_{p} = \frac{ (\lambda_{1} - \lambda_{2}) \sin 2 \alpha}{ C (p-1) -  \cos 2 \alpha} + O \bigg( \frac{1}{p^{2}} \bigg).
\end{equation*}
Now, we set $C_{1}= (\lambda_{1} - \lambda_{2}) \sin 2 \alpha$, $C_{2}=C=  r^{-1} |q|^{2}$ and $C_{3} =  -2C  - 2 \cos 2 \alpha$,
then  we can rewrite $\theta_{p}$ in the following more compact form:
\begin{equation}
\label{genius}
\theta_{p} = \frac{C_{1}}{ C_{2} p + C_{3}} = \frac{C_{1}}{C_{3}( \frac{C_{2}}{C_{3}} p + 1)} = \frac{C_{1}}{C_{3}} \frac{1}{\frac{C_{2}}{C_{3}} p + 1}.
\end{equation}
Finally set $x=\frac{1}{\frac{C_{2}}{C_{3}} p}$, to conclude we need to need only to apply the Taylor's expansion, near $x=0$,  for the function $\frac{1}{\frac{1}{x}+1}= \frac{x}{x+1}$ (that is
$
\frac{x}{x+1} = x + O(x^{2})
$), then  \eqref{genius} can be rewritten as 
\begin{equation*} 
\theta_{p} = \frac{C_{1}}{C_{3}} \bigg [ \frac{C_{3}}{C_{2} p } + O \bigg( \frac{1}{p^{2}} \bigg) \bigg] = \overline{C}\, \frac{1}{ p} + O \bigg( \frac{1}{p^{2}} \bigg),
\end{equation*}
with $\overline{C}=\frac{C_{1}}{C_{2} } $. So $\beta=\overline{C} +O \bigg( \frac{1}{p} \bigg)$.\\

\end{proof}

Let us now make  a few remarks on the previous result on some special cases. 
\begin{rem}[Case $\lambda_1=\lambda_2$.]
\label{lambdaEqual}
Note that $\lambda_1$ and $\lambda_2$ are the eigenvalues of the matrix of the second order derivatives. 
Whenever $\lambda_1=\lambda_2=:\lambda$, then 
$$
f_p(\theta)=-(p-1) r^{-1}|q|^2\sin^2\theta+\lambda,
$$
the obviously, for all $p>1$, the maximum is attained for $\theta=0$, This implies that in this case, at non characteristic points (i.e. $|q|\neq 0$), the optimal control for the $p$-problem is actually the same of the optimal control for the limit problem $p=+\infty$,
i.e.
$$\overline{\nu}=
I_{2} -  n_{0} \otimes n_{0},
$$
where $n_0$ is the horizontal normal, see \cite{dirr} for more details for the case $p=+\infty$.

\end{rem}

\begin{rem}[Case $\alpha=0$ or $\alpha=\frac{\pi}{2}$.]
If $\alpha\in \left\{0,\frac{\pi}{2}\right\}$ again the maximum is  attained in $\theta=0$.
Remember that the case $\alpha=0$ corresponds to the case case where the horizontal gradient in $\R^2$ point is $(1,0)^t$, then the horizontal normal points in the direction of the first vector field $X_1$, while in the case
$\alpha=\frac{\pi}{2}$  the horizontal normal points in the direction of the second vector field $X_2$.
\end{rem}

To conclude this section let give the general result, removing the additional assumption of $M$ diagonal.

\begin{theorem} \label{tubBis}
Let us consider the $p$-Hamiltonian $H_p$ 
 introduced 
in \eqref{hami21}.
 Fixed $r$, $q$  and $M$ and assume that  $q \neq 0$
 Then, for large $p$, the optimal control  is $\overline{\nu}=\nu_{\overline{\theta}}$,  where  $\nu_{\theta}$ is defined in \eqref{ca2} and
\begin{equation}
\overline{\theta} = \overline{C}\; \frac{1}{ p} + O \bigg( \frac{1}{p^{2}} \bigg), 
\end{equation}
  with the constant  $\overline{C}$ depends only 
on the variable $r$, $q$ and the eigenvalues of the matrix $M$. Note that the expansion is only valid for $p|q|^2$ large.
For $p|q|^2=O(\lambda_1)$ see Chapter 7.
\end{theorem}

\begin{proof} This follows immediately by Theorem \ref{tub} and Remark
\ref{diagonalMatrix}.

\end{proof}

\subsection{Case $q = 0$: characteristic points.}

In the case that $q=0$ the function $h_{p}$ is simplified and it does not depend on $p$. In fact it has the form
\begin{equation*}
h_{p}(r,0,M, \nu) = h(M, \nu) = Tr[ \nu \nu^{T} M].
\end{equation*}
Furthermore we observe that we can generate a general admissible control starting from the rotation of a generic unit vector. For sake of simplicity (see Remark \ref{RemarkStartingVector_q=0} later)
we fix as starting vector 
\begin{equation*}
n_{\infty} := \begin{bmatrix} 1 \\ 0 \end{bmatrix} ,
\end{equation*}
and we rotate it by a rotation matrix $R(\theta)$ as follows:
\begin{equation*}
n_{\theta} := R(\theta)n_{\infty}=
\begin{bmatrix} \cos \theta & -\sin \theta \\ \sin \theta & \cos \theta \end{bmatrix} \begin{bmatrix} 1 \\ 0 \end{bmatrix} = \begin{bmatrix} \cos \theta \\ \sin \theta \end{bmatrix},
\end{equation*}
where $\theta \in [0, 2 \pi)$. We define the admissible control as
\begin{equation}
\label{FInalControl=0}
\nu_{\theta}: = I_{2} - n_{\theta} \otimes n_{\theta} = \begin{bmatrix} \sin^{2} \theta & - \sin \theta \cos \theta \\ - \sin \theta \cos \theta & \cos^{2} \theta \end{bmatrix}.
\end{equation}
Then, by recalling that $\nu_{\theta}$ is a projection matrix we obtain immediately
$$
h_{p}(r,0,M, \nu_{\theta}) = \lambda_{1} \sin^{2}\theta  + \lambda_{2} \cos^{2}\theta. 
$$

\begin{theorem} \label{tub8}
Let us consider the $p$-Hamiltonian introduced as in \eqref{hami21}, $\nu_{\theta}$ defined as in \eqref{ca2} and fixed $r$, $q$  and $M$. Let us denote by $\lambda_1,\lambda_2$  the eigenvalues of $M$. Assume that  $q = 0$. Then, for all $p>1$, the optimal control for the $p$-Hamiltonian is independent on $p$ and it is given by $\overline{\nu} = \nu_{\overline{\theta}}$ where the control is defined in \eqref{FInalControl=0} and
\begin{equation*}
\overline{\theta} = \frac{ \pi}{2} \ \ \mbox{whenever  $\lambda_{1} - \lambda_{2}>0$},
\end{equation*}
\begin{equation*}
\overline{\theta}=0\ \ \mbox{whenever $\lambda_{1} - \lambda_{2}<0$},
\end{equation*}
while  for $\lambda_{1} = \lambda_{2}=:\lambda$, $h_p(r,0,M, \nu_{\theta})=\lambda$ is constant so all possible angle $\theta$ are associated to optimal controls.
\end{theorem}.
\begin{proof}
 First recall that by Remark
\ref{diagonalMatrix} we can assume that $M$ has the  diagonal form given in \eqref{bre1}.
Fixed $r$ and $m$ and denoting 
$h_p(r,0,M, \nu_{\theta})=:f(\theta)$, we look at the stationary points for the case $\lambda_1\neq \lambda_2$.
Taking the derivatives, we have
\begin{align}
f'(\theta)= 2\lambda_{1} \sin \theta \cos \theta - 2 \lambda_{2} \sin \theta  \cos \theta  = 2 (\lambda_{1} - \lambda_{2})  \sin \theta  \cos \theta \nonumber \ \ \ \ \ \ \ \  \\   = (\lambda_{1} - \lambda_{2}) \sin  2 \theta. \ \ \ \ \ \ \ \ \ \ \ \ \
\end{align}
If $\lambda_{1}= \lambda_{2}$ we obtain that $f'(\theta)=0$ for all $\theta \in [0, 2 \pi) $, i.e. the function $f$ is constant. If $\lambda_{1}\neq \lambda_{2}$ we note that
\begin{equation}
\sin 2\theta = 0 \Longleftrightarrow \theta = \frac{k \pi}{2} \ \ \mbox{with $k \in \{ 0,1, 2, 3 \}$}.
\end{equation}
In order to find the maximum values for $\lambda_1\neq \lambda_2$ we can easily compute the second derivatives, that is
\begin{equation}
f''(\theta) = 2 (\lambda_{1} - \lambda_{2}) \cos  2 \theta.
\end{equation}
so the stationary points in $[0,\pi)$ are $\theta=0$ and $\theta=\frac{\pi}{2}$.
Then
\begin{enumerate}
\item for $\lambda_{1} - \lambda_{2} >0$ it holds true $f''(\frac{\pi}{2})  = -2 (\lambda_{1} - \lambda_{2})<0$.
\item for $\lambda_{1} - \lambda_{2} <0$ it holds true $f''(0) = -2 (\lambda_{1} - \lambda_{2})<0$,
\end{enumerate}
and this concludes the proof.

\end{proof}

\begin{rem}
\label{RemarkStartingVector_q=0}
An easy computation shows that the optimal control is independent on the choice of the  staring vector  $n_{\infty}$.
In fact, choosing a generic unit vector $n_{\infty}=(\cos\alpha, \sin\alpha)^T$ we would have got the control in the form given in   \eqref{ca2} and the same results found in Theorem \ref{tub8} for the angle $\theta+\alpha$, then the  optimal control would be exactly the same.
\end{rem}

\begin{rem}\label{Rem64}
It is instructive to compare this with the formula derived in \cite{dirr}. There, in characteristic points,  we get a projection on the on the  eigenspace of the minimal eigenvalue for subsolutions and on that of the maximal eigenvalue for supersoutions. Here we see: The optimal control in a characteristic point is not unique if both eigenvalues are equal, 
otherwise it is the projection on the eigenspace of the {\em maximal } eigenvalue.
\end{rem}


\section{Numerical computations and illustrations}
In this section we give some computed examples for the optimal control to illustrate the behaviour of the controlled random walk, in particular in the vicinity of critical points.

Consider the $\pi$-periodic function $f_p$ from \eqref{LargeP}. First note that $q$ and $p$ appear only as $p|q|^2,$ so $f_p$ can be expressed as a function
of $p|q|^2.$ This means that in the following graphs the limit $p\to \infty$ corresponds to a scaling in the $|q|$-direction. Moreover $|q|\to 0$ means moving towards a characteristic point, $p\to \infty$ away from it. As $r^{-1}$ appears only multiplying $p$, we see that local convergence does not depend on the choice of the level set function. This is to be expected, as the limit evolution is geometric.
Moreover, note that $r,$ the value function, is constant on a level set.
First let us consider Figure 1, which plots $f_p$ for $p=10,$ $\alpha =0,$ 
$\lambda_1=1,$ $\lambda_2=0$ and $r=1.$ The  $q_1$-axis points in the direction of  eigenvector corresponding to the eigenvalue $\lambda_1.$ Geometrically this means that the projected horizontal eigenvector points exactly in the direction of the largest eigenvalue of the horizontal Hessian. In this case we see that for a distinct value of $q$ the maximum jumps from $0$ to $\pi/2,$ corresponding to either maximizing the second term in $f_p$, i.e. making $|\sin(\theta)|=1,$
or maximizing the first term, making $\sin(\theta)=0.$
\begin{center}
\begin{figure}\label{figure1}
\includegraphics[scale=0.32]{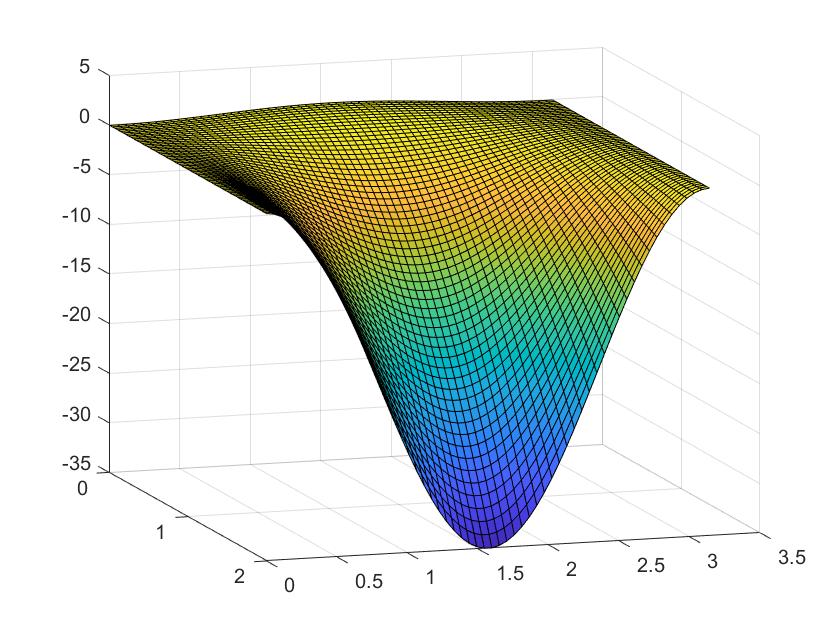}
\caption{
$f_{10}(\theta)$ for $p=10$ and $\alpha=0.$
}
\end{figure}
\end{center}

\newpage
\vspace{-2cm}
\begin{center}
\begin{figure}\label{figure2}
\begin{tabular}{lr}
\hspace{-0.8cm}
\includegraphics[scale=0.23]{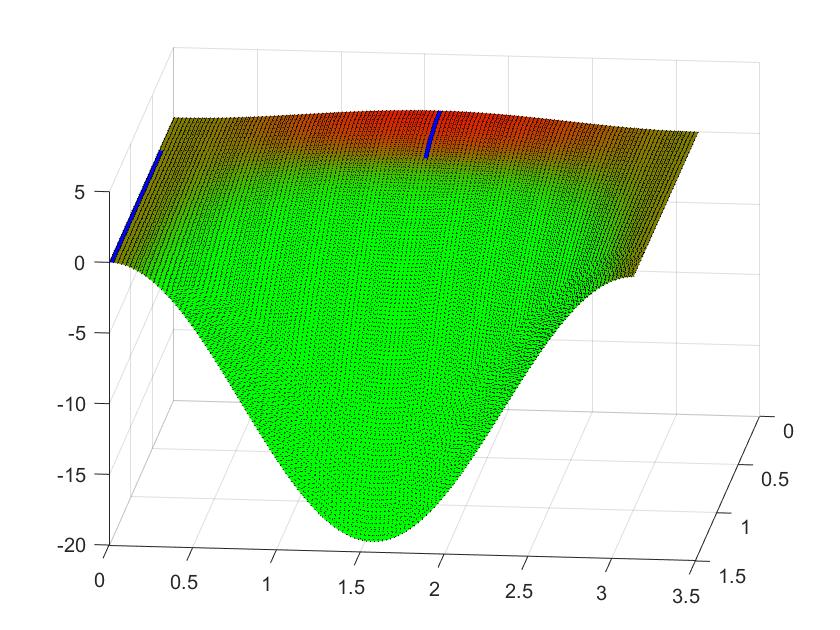}
\hspace{-0.9cm}
\includegraphics[scale=0.24]{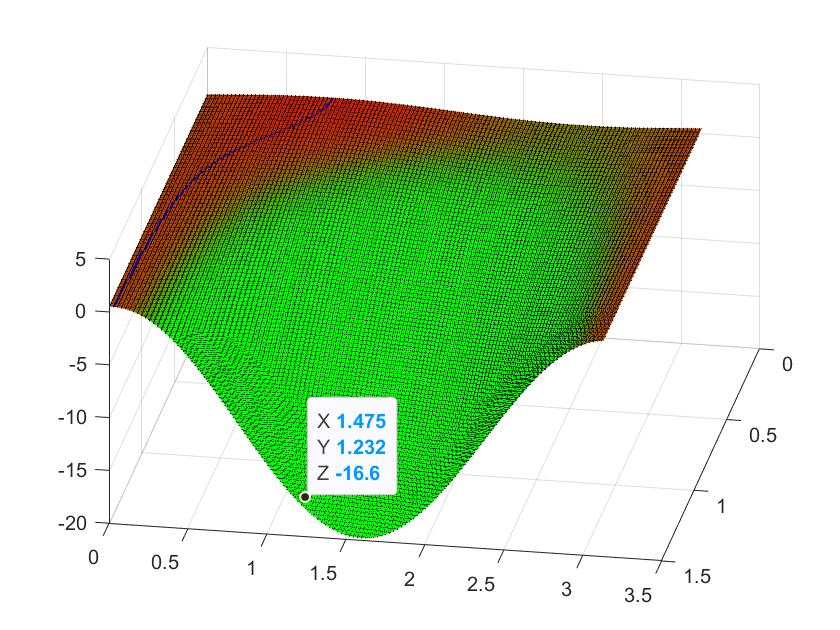}
\end{tabular}
\caption{ On the left
$f_{10}(\theta)$  (i.e. $p=10$) for $\alpha=0,$ on the right for $\alpha=\pi/4,$ the branch of maximizing $\theta$ is in both cases in blue. Note that the maximising angle is discontinuous in the left picture and  continuous on the right picture. 
}
\end{figure}
\end{center}
For $\alpha\not=0,$ however, there is a continuous branch of maximizing angles. The situation is illustrated in Figure 2.\\

It turns out that the case on the left, where the $q$-vector points in the direction of the eigenvector with the largest eigenvalue of the horizontal Hessian, is the only case where such a singularity occurs. Here, in the limit $p\to \infty,$   for the critical point, the optimal control is projection on the eigenspace for the maximal eigenvalue, i.e. parallel to $q,$ while immediately away  from the critical point the optima control projects {\em orthogonally} to $q.$ \\

We see in Figure 3 that  this vortex-like discontinuity occurs for $\alpha=\pi$ (equivalent to $\alpha=0$) for small (non vanishing) $|q|.$

\newpage
\begin{center}
\begin{figure}\label{figure3}
\includegraphics[scale=0.38]{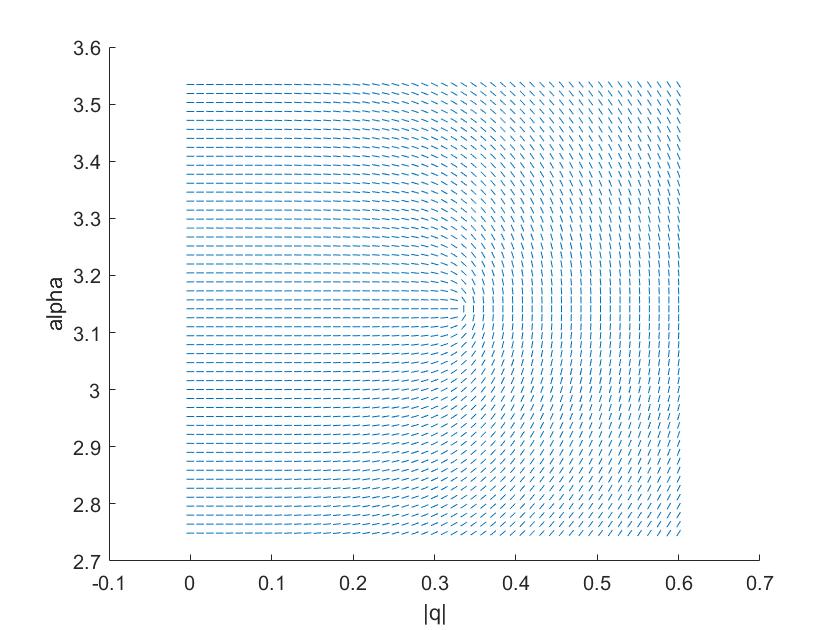}
\caption{ Direction of optimal control (i.e. allowed direction of motion of the controlled process) for $p=10,$ $M={\rm diag}(1,0).$ Note the singularity when $|q|^2p=\lambda_1-\lambda_2=1.$ In this regime, the asymptotic expansion of  Section 6 is not valid.
}
\end{figure}
\end{center}
\vspace{-1cm}
For $\lambda_1=\lambda_2,$ this discontinuity is moved into the characteristic point, see below the case of the unit sphere (see Figure 5).\\

Finally, in order to illustrate the convergence of the optimal control, we plot in Figure 4 for the same parameters an entire period for both $p=10$ and $p=50.$
\vspace{-1cm}
\begin{center}
\begin{figure}\label{figure4}
\begin{tabular}{lr}
\hspace{-1cm}
\includegraphics[scale=0.25]{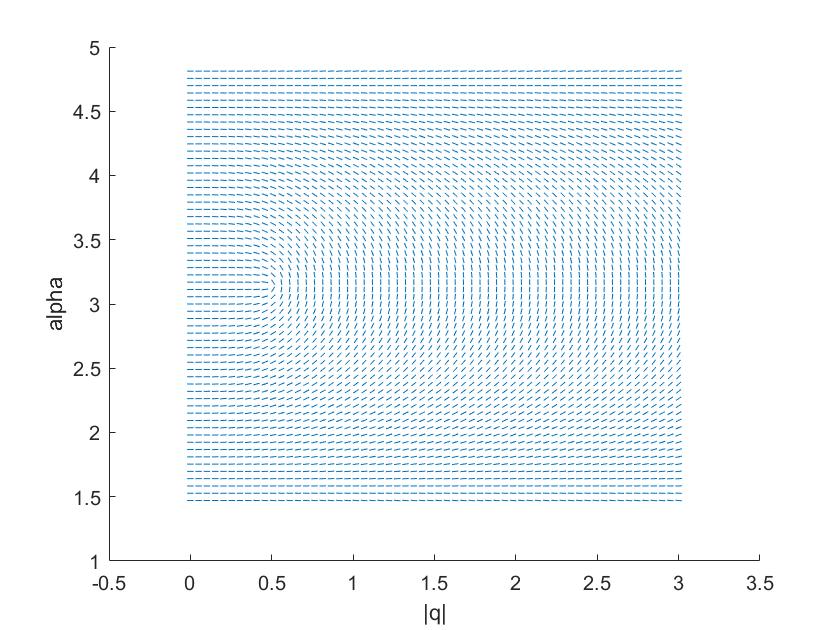}
\hspace{-0.9cm}
\includegraphics[scale=0.25]{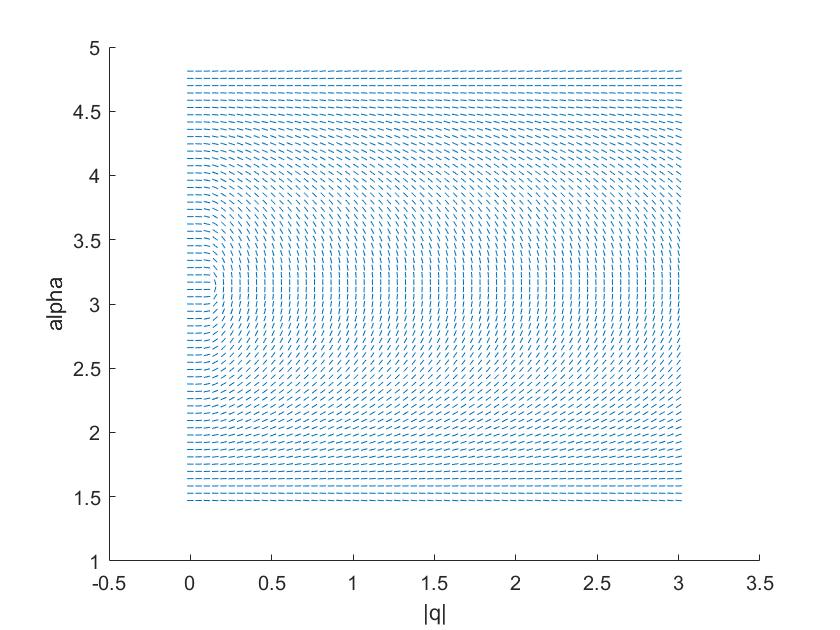}
\end{tabular}
\caption{ Direction of optimal control (i.e. allowed direction of motion of the controlled process) for  $M={\rm diag}(1,0)$, i.e. $\lambda_1=1$ and $\lambda_2=0$. 
On the left picture: $p=5$; while on the right picture: $p=30.$}
\end{figure}
\end{center}
\newpage

Let us apply this to two specific surfaces, the unit sphere centred at $(0,0,1)$ and and ellipsoid with same center, but given by
$2x^2+y^2+(z-1)^2=1.$ Due to the lower symmetry, the singularity of the control field  for the ellipsoid is moved away from the characteristic point and clearly visible. We have plotted the horizontal normal as dashed line and the control field as solid line, both in a 3-dimensional perspective and the projection of the vectors on the $x-y$-plane.
\begin{center}
\begin{figure}\label{figure5}
\begin{tabular}{lr}
\hspace{-1cm}
\includegraphics[scale=0.25]{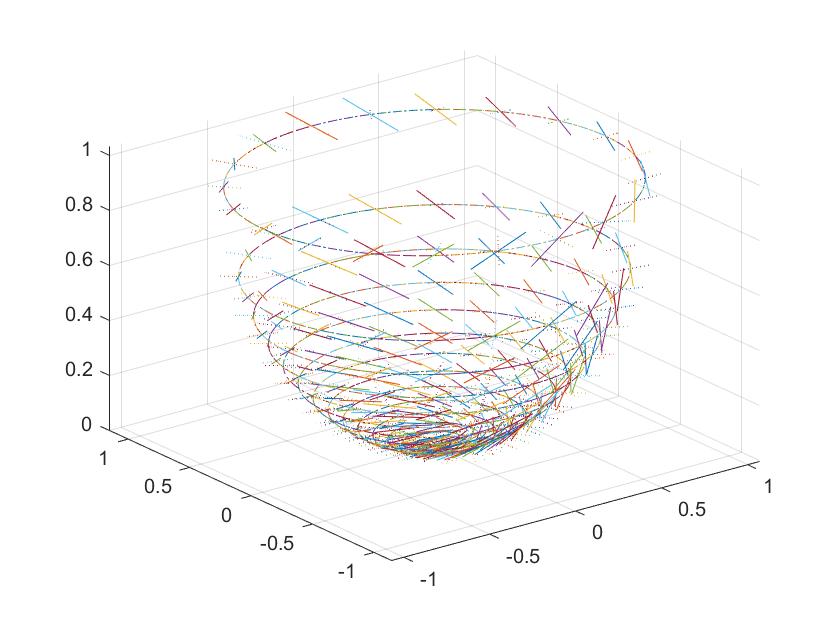}
\hspace{-0.9cm}
\includegraphics[scale=0.25]{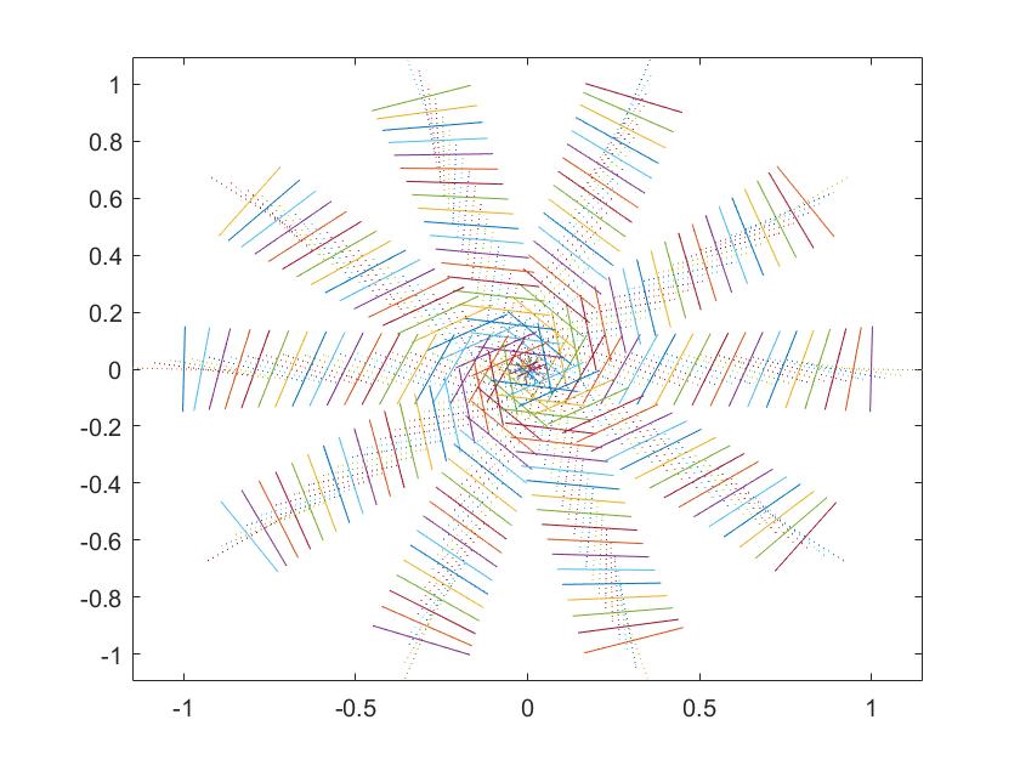}
\end{tabular}
\caption{Direction of optimal control (solid) and horizontal normal (dashed) for the unit sphere with characteristic point $(0,0,0).$ On the left these vectors are represented in 3D, while on the right we see their  2D-projection on the $x-y$-plane (for $p=5$).}
\end{figure}
\end{center}

\begin{center}
\begin{figure}\label{figure6}
\begin{tabular}{lr}
\hspace{-1cm}
\includegraphics[scale=0.25]{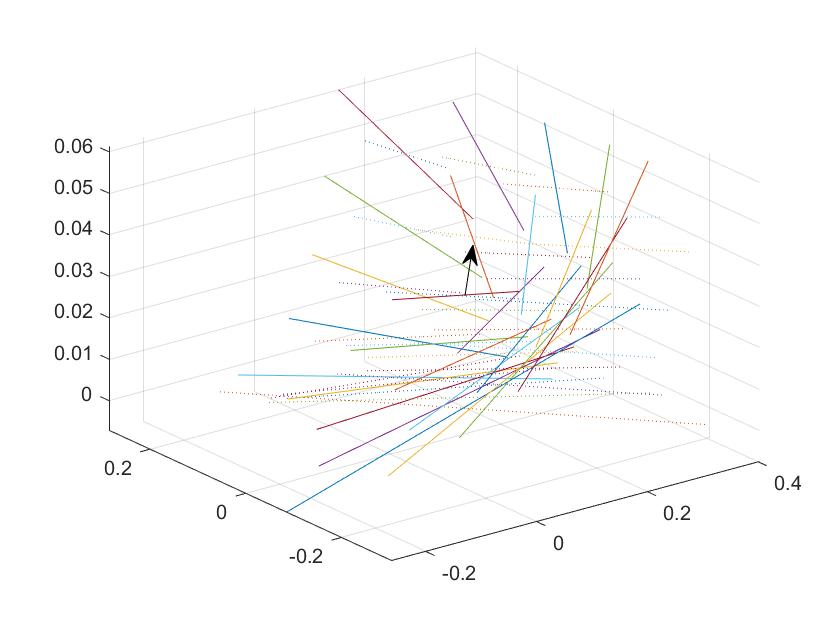}
\hspace{-0.9cm}
\includegraphics[scale=0.25]{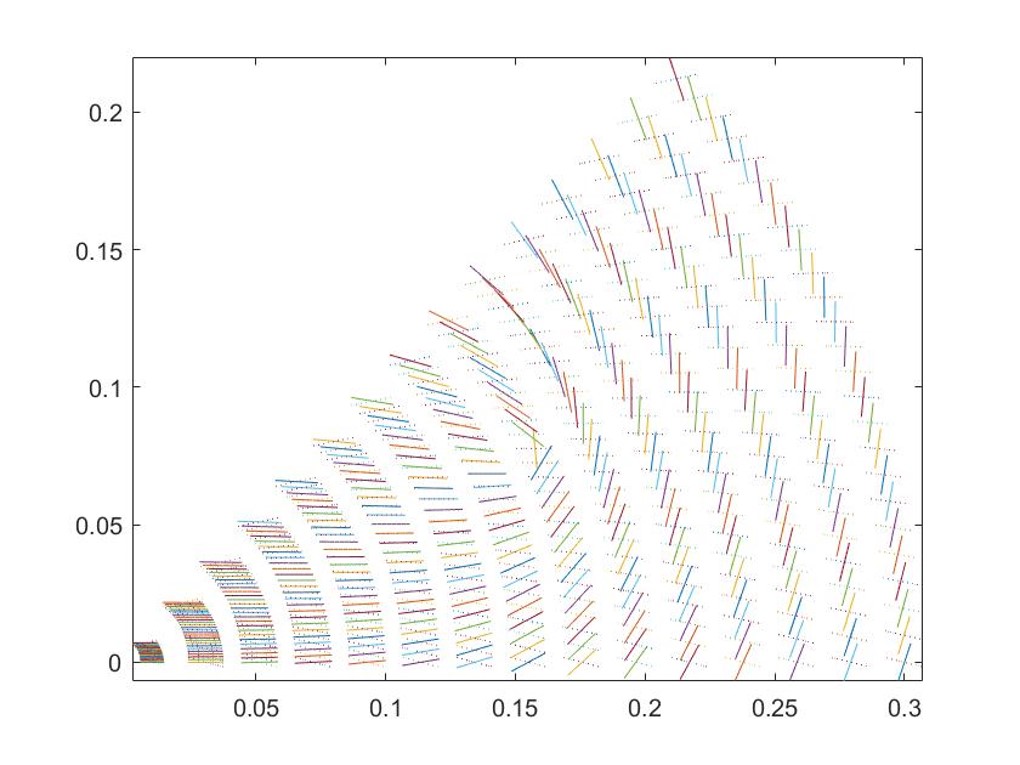}
\end{tabular}
\caption{Direction of optimal control (solid) and horizontal normal (dashed) for the ellipsoid $2x^2+y^2+(z-1)^2=1$ with characteristic point $(0,0,0).$ On the left these vectors are represented in 3D, while on the right we see their  2D-projection on the $x-y$-plane (for $p=5$). The arrow points out the singularity. At this point, the horizontal normal points in the direction of an eigenvector of the horizontal Hessian.}
\end{figure}
\end{center}

\FloatBarrier

\end{document}